\documentclass[reqno]{amsart}
\usepackage{hyperref}

\begin{document}
\title[Hermite-Hadamard, Hermite-Hadamard-Fej\'{e}r, Dragomir-Agarwal and ...]
{Hermite-Hadamard, Hermite-Hadamard-Fej\'{e}r, Dragomir-Agarwal
and Pachpatte type inequalities for convex functions via new
fractional integrals}

\author[B.~Ahmad, A.~Alsaedi, M.~Kirane and
B.\,T.~Torebek\hfil \hfilneg] {Bashir~Ahmad, Ahmed~Alsaedi, Mokhtar~Kirane and
Berikbol~T.~Torebek}  

\address{Bashir Ahmad \newline
NAAM Research Group, Department of Mathematics, \newline Faculty of Science, King Abdulaziz University, \newline P.O. Box 80203, Jeddah 21589, Saudi Arabia}
\email{bashirahmad\_qau@yahoo.com}

\address{Ahmed Alsaedi \newline
NAAM Research Group, Department of Mathematics, \newline Faculty of Science, King Abdulaziz University, \newline P.O. Box 80203, Jeddah 21589, Saudi Arabia}
\email{aalsaedi@hotmail.com}

\address{Mokhtar Kirane \newline
LaSIE, Facult\'{e} des Sciences, \newline Pole Sciences et Technologies, Universit\'{e} de La Rochelle, \newline Avenue M. Crepeau, 17042 La Rochelle Cedex, France \newline NAAM Research Group, Department of Mathematics, \newline Faculty of Science, King Abdulaziz University, \newline P.O. Box 80203, Jeddah 21589, Saudi Arabia}
\email{mkirane@univ-lr.fr}

\address{Berikbol T. Torebek \newline
Al--Farabi Kazakh National University \newline
al--Farabi ave. 71, 050040, Almaty, Kazakhstan \newline Institute of
Mathematics and Mathematical Modeling.\newline 125 Pushkin str.,
050010 Almaty, Kazakhistan} \email{torebek@math.kz}

\begin{abstract}
The aim of this paper is to establish Hermite-Hadamard, Hermite-Hadamard-Fej\'{e}r, Dragomir-Agarwal and Pachpatte type inequalities for new fractional integral operators with exponential kernel. These results allow us to obtain a new class of functional inequalities which generalizes known inequalities involving convex functions. Furthermore, the obtained results may act as a useful source of inspiration for future research in convex analysis and related optimization fields. \end{abstract}

\maketitle \numberwithin{equation}{section}
\newtheorem{theorem}{Theorem}[section]
\newtheorem{corollary}[theorem]{Corollary}
\newtheorem{lemma}[theorem]{Lemma}
\newtheorem{remark}[theorem]{Remark}
\newtheorem{problem}[theorem]{Problem}
\newtheorem{example}[theorem]{Example}
\newtheorem{definition}[theorem]{Definition}
\allowdisplaybreaks

\section{Introduction}

The inequalities for convex functions due to Hermite and Hadamard are found to be of great importance, for example, see \cite{DP00, PPT92}. According to the inequalities \cite{H1893, H1883},  \begin{itemize}
 \item if $u:\,I\rightarrow \mathbb{R}$ is a convex function on the interval $I\subset \mathbb{R}$ and $a,b\in I$ with $b>a,$ then
\begin{equation}\label{1.1}u\left(\frac{a+b}{2}\right)\leq \frac{1}{b-a}\int\limits^b_a u(y)dy \leq \frac{u(a)+u(b)}{2}.\end{equation}
 \end{itemize}
For a concave function $u$, the inequalities in \eqref{1.1} hold in the reversed direction. We note that Hadamard's inequality refines the concept of convexity,  and it  follows from Jensen's inequality. The classical Hermite-Hadamard inequality yields estimates for the mean value of a continuous convex function $u : [a, b] \rightarrow \mathbb{R}.$ The well-known inequalities dealing with the integral mean of a convex function $u$ are the Hermite-Hadamard inequalities or its weighted versions. They are also known as Hermite-Hadamard-Fej\'{e}r inequalities.

In \cite{F06}, Fej\'{e}r obtained the  weighted generalization of Hermite-Hadamard inequality \eqref{1.1} as follows.
\begin{itemize}
 \item Let $u:\, [a,b]\rightarrow \mathbb{R}$ be a convex function. Then the inequality
\begin{equation}\label{1.2}u\left(\frac{a+b}{2}\right)\int\limits^b_a w(y)dy\leq\int\limits^b_a u(y) w(y)dy \leq \frac{u(a)+u(b)}{2}\int\limits^b_a w(y)dy \end{equation} holds for a nonnegative, integrable function $v:\, [a,b]\rightarrow \mathbb{R}$, which is symmetric to $\frac{a+b}{2}.$
 \end{itemize}

In \cite{DA98}, Dragomir and Agarwal obtained the following results in connection with the right part of \eqref{1.1}:
\begin{itemize}
 \item Let $u:\, I\subseteq \mathbb{R}\rightarrow \mathbb{R}$ be a differentiable mapping on $I,a,b \in I.$ If $|u'|$ is convex on $[a,b],$ then the following inequality holds: \begin{equation}\label{1.3}\left|\frac{u(a)+u(b)}{2}-\frac{1}{b-a}\int\limits^b_a u(y)dy\right|\leq \frac{b-a}{8}\left(|u'(a)|+|u'(b)|\right).
\end{equation}
 \end{itemize}

In \cite{P03}, Pachpatte established two new Hermite-Hadamard type inequalities for products of convex functions as follows:
\begin{itemize}
 \item Let $u$ and $w$ be nonnegative and convex functions on $[a,b]\subset \mathbb{R},$ then \begin{equation}\label{1.4}\begin{aligned}\frac{1}{b-a}\int\limits^b_a u(y)w(y)dy \leq \frac{u(a)w(a)+u(b)w(b)}{3} + \frac{u(a)w(b)+u(b)w(a)}{6}
\end{aligned}\end{equation} and
\begin{equation}\begin{aligned} 2 &u\left(\frac{a+b}{2}\right) w\left(\frac{a+b}{2}\right) \leq \frac{1}{b-a}\int\limits^b_a u(y)w(y)dy \\& +\frac{u(a)w(a)+u(b)w(b)}{6}  + \frac{u(a)w(b)+u(b)w(a)}{3}.\label{1.5}
\end{aligned}\end{equation}\end{itemize}

Next we present some results on the generalization of aforementioned inequalities.

In \cite{SSYB13}, Sarikaya et. al. represented Hermite-Hadamard and Dragomir-Agarwal inequalities in fractional integral forms as follows.
\begin{itemize}
 \item Let $u:\,[a,b]\rightarrow \mathbb{R}$ be a positive function and $u\in L^1([a,b]).$ If $u$ is a convex function on $[a,b],$ then the following inequalities for fractional integrals hold
\begin{equation*}u\left(\frac{a+b}{2}\right)\leq \frac{\Gamma(\alpha+1)}{2(b-a)^\alpha}\left[I^\alpha_au(b)+I^\alpha_bu(a)\right] \leq \frac{u(a)+u(b)}{2}\end{equation*} with $\alpha>0.$
\end{itemize}
\begin{itemize}
 \item Let $u:\,[a,b]\rightarrow \mathbb{R}$ be a differentiable mapping on $(a, b).$ If $|u'|$ is convex on [a, b], then the following inequality for fractional integrals holds:
     \begin{multline*}\left|\frac{u(a)+u(b)}{2}-\frac{\Gamma(\alpha+1)}{2(b-a)^\alpha}\left[I^\alpha_au(b)+I^\alpha_bu(a)\right]\right|\\ \leq \frac{b-a}{2(\alpha+1)}(1-2^{-\alpha})\left(|u'(a)|+|u'(b)|\right).\end{multline*}
\end{itemize}
In \cite{I16}, I\c{s}can obtained the following Hermite-Hadamard-Fej\'{e}r integral inequalities via fractional integrals:
\begin{itemize}
 \item  Let $u:\,[a,b]\rightarrow \mathbb{R}$ be convex function with $a < b$ and $u \in L^1([a, b]).$ If $v:\,[a,b]\rightarrow \mathbb{R}$ is nonnegative, integrable and symmetric to $(a+b)/2,$ then the following inequalities for fractional integrals hold \begin{align*}u\left(\frac{a+b}{2}\right)\left[I^\alpha_av(b)+I^\alpha_bv(a)\right]&\leq \left[I^\alpha_a(uv)(b)+I^\alpha_b(uv)(a)\right]\\& \leq \frac{u(a)+u(b)}{2}\left[I^\alpha_av(b)+I^\alpha_bv(a)\right] \end{align*} with $\alpha>0.$
\end{itemize}
Many generalizations and extensions of the Hermite-Hadamard, Hermite-Hadamard-Fej\'{e}r, Dragomir-Agarwal and Pachpatte type inequalities were obtained for various classes of functions using fractional integrals; see \cite{BPP16, C16,  CK17, HYT14, ITM16, I16,  JS16, SSYB13, WLFZ12, ZW13} and references therein.

These studies motivated us to consider a new class of functional inequalities for convex functions generalizing the classical Hermite-Hadamard, Hermite-Hadamard-Fej\'{e}r, Dragomir-Agarwal and Pachpatte inequalities. Here we emphasize that we derive some functional inequalities for the new fractional integral operators with exponential kernel. The difference between our results and the known generalizations is that the above fractional analogues of functional inequalities do not follow from our results. In fact our results are the simplest generalizations of only classical inequalities.

The paper is organized as follows. Section \ref{Prel} contains some basic concepts related to our proposed study. In Section \ref{HH}, a Hermite-Hadamard type inequality for a fractional integral with an exponential kernel is proved. The fractional analogue of the Hermite inequality is investigated in Section \ref{HHF}. Section \ref{DA} is devoted to the generalization of Dragomir-Agarwal's inequality. In Section \ref{P}, we obtain generalized Pachpatte-type inequalities with fractional integrals in the class of convex functions.

\section{Preliminaries}\label{Prel}
We give some definitions for further use.
\begin{definition}\label{def1.1} A function $u : [a, b]\subset \mathbb{R} \rightarrow \mathbb{R}$ is said to be convex if $$u(\mu x+(1-\mu)y)\leq \mu u(x)+(1-\mu)u(y)$$ for all $x, y\in [a,b]$ and $\mu\in [0,1].$ We call $u$ a  concave function if $(-u)$ is convex. \end{definition}

Now we give some necessary concepts related to the new fractional integral which are used in the sequel.

\begin{definition}\label{def1.2} Let $f\in L_1(a,b).$ The fractional integrals $\mathcal{I}^\alpha_a$ and $\mathcal{I}^\alpha_b$ of order $\alpha\in (0,1)$ are defined by
\begin{equation} \label{I-1}
\mathcal{I}^\alpha_a u(x)=\frac{1}{\alpha}\int\limits^x_a \exp\left(-\frac{1-\alpha}{\alpha}(x-s)\right) u(s)ds,\, x>a
\end{equation}
and
\begin{equation}\label{I-2}
\mathcal{I}^\alpha_b u(x)=\frac{1}{\alpha}\int\limits^b_x \exp\left(-\frac{1-\alpha}{\alpha}(s-x)\right) u(s)ds,\, x<b
\end{equation}
 respectively.
\end{definition}

If $\alpha=1,$ then $$\lim_{\alpha\rightarrow 1}\mathcal{I}^\alpha_a u(x)=\int\limits^x_a u(s)ds,\,\, \lim_{\alpha\rightarrow 1}\mathcal{I}^\alpha_b u(x)=\int\limits^b_x u(s)ds.$$
Moreover, in view of
$$\lim_{\alpha\rightarrow 0} \frac{1}{\alpha}\exp\left(-\frac{1-\alpha}{\alpha}(x-s)\right)=\delta(x-s),$$
we deduce that
$$\lim_{\alpha\rightarrow 0}\mathcal{I}^\alpha_a u(x)=u(x),\,\, \lim_{\alpha\rightarrow 0}\mathcal{I}^\alpha_b u(x)=u(x).$$

\begin{definition} The left and right Riemann--Liouville
fractional integrals $I_{a} ^\alpha$ and $I_{b} ^\alpha$ of order $\alpha\in\mathbb R$ ($\alpha>0$) are given by
$$
I_{a} ^\alpha  \left[ f \right]\left( t \right) = {\rm{
}}\frac{1}{{\Gamma \left( \alpha \right)}}\int\limits_a^t {\left(
{t - s} \right)^{\alpha  - 1} f\left( s \right)} ds, \,\,\, t\in(a,b],
$$
and
$$ I_{b}
^\alpha  \left[ f \right]\left( t \right) = {\rm{
}}\frac{1}{{\Gamma \left( \alpha \right)}}\int\limits_t^b {\left(
{s - t} \right)^{\alpha  - 1} f\left( s \right)} ds, \,\,\, t\in[a,b),
$$
respectively. Here $\Gamma$ denotes the Euler gamma function.
\end{definition}
We henceforth set $\rho=\frac{1-\alpha}{\alpha}(b-a).$

\section{Hermite-Hadamard type inequality}\label{HH}

\begin{theorem}\label{th2.1} Let $u: [a, b]\rightarrow \mathbb{R}$ be a positive function with $0\leq a < b$ and $u \in L_1 (a, b).$ If $u$ is a convex function on $[a, b],$ then the following inequalities for fractional integrals \eqref{I-1} and \eqref{I-2} hold:
\begin{equation}\label{2.1} u\left(\frac{a+b}{2}\right)\leq \frac{1-\alpha}{2\left(1-\exp\left(-\rho\right)\right)}\left[\mathcal{I}^\alpha_a u(b)+\mathcal{I}^\alpha_b u(a)\right]\leq \frac{u(a)+u(b)}{2}.
\end{equation}
\end{theorem}

\begin{proof} Since $u$ is a convex function on $[a,b],$ we get for $x$ and $y$ from $[a, b]$ with $\mu=\frac{1}{2}$ \begin{equation}\label{2.2}u\left(\frac{x+y}{2}\right) \leq \frac{u(x)+u(y)}{2},\end{equation} which, for $x=ta+(1-t)b,\, y=(1-t)a+tb,$ takes the form:
\begin{equation}\label{2.3}2u\left(\frac{a+b}{2}\right)\leq u(ta+(1-t)b)+u((1-t)a+tb).\end{equation} Multiplying both sides of \eqref{2.3} by $\exp\left(-\rho t\right)$ and  then integrating the resulting inequality with respect to $t$ over $[0,1],$ we obtain
\begin{align*}\frac{2\left(1-\exp\left(-\rho \right)\right)}{\rho}u\left(\frac{a+b}{2}\right) & \leq\int\limits^1_0 \exp\left(-\rho t\right) \left[u(ta+(1-t)b)+u((1-t)a+tb)\right]dt\\&= \int\limits^1_0 \exp\left(-\rho t\right) u(ta+(1-t)b)dt \\&+\int\limits^1_0 \exp\left(-\rho t\right)u((1-t)a+tb)dt\\&=\frac{1}{b-a}\int\limits^b_a \exp\left(-\frac{1-\alpha}{\alpha}(b-s)\right)u(s)ds\\&+ \frac{1}{b-a}\int\limits^b_a \exp\left(-\frac{1-\alpha}{\alpha}(s-a)\right)u(s)ds\\&= \frac{\alpha}{b-a}\left[\mathcal{I}^\alpha_a u(b)+ \mathcal{I}^\alpha_b u(a)\right].
\end{align*} As a result, we get
$$\frac{2\left(1-\exp\left(-\rho \right)\right)}{\mathcal{A}}u\left(\frac{a+b}{2}\right) \leq \frac{\alpha}{b-a}\left[\mathcal{I}^\alpha_a u(b)+ \mathcal{I}^\alpha_b u(a)\right].$$ Thus the first inequality of \eqref{2.1} is established.

For the proof of the second inequality in \eqref{2.1},  we first note that if $u$ is a convex function, then, for $t\in [0, 1],$ it yields
$$u(ta+(1-t)b)\leq tu(a)+(1-t)u(b)$$ and $$u((1-t)a+tb)\leq (1-t)u(a)+tu(b).$$ By adding the above two inequalities,  we have
\begin{equation}\label{2.4}u(ta+(1-t)b)+u((1-t)a+tb)\leq u(a)+u(b).
\end{equation} Multiplying both sides of \eqref{2.4} by $\exp\left(-\rho t \right)$ and integrating the resulting inequality with respect to $t$ over $[0,1],$ we obtain \begin{align*}\frac{2\left(1-\exp\left(-\rho \right)\right)}{\rho}\left[u(a)+u(b)\right] & \geq \int\limits^1_0 \exp\left(-\rho t\right) u(ta+(1-t)b)dt \\&+\int\limits^1_0 \exp\left(-\rho t\right)u((1-t)a+tb)dt,\end{align*} that is,  $$\frac{\alpha}{b-a}\left[\mathcal{I}^\alpha_a u(b)+ \mathcal{I}^\alpha_b u(a)\right]\leq \frac{2\left(1-\exp\left(-\rho \right)\right)}{\rho}\left[u(a)+u(b)\right].$$ Hence the second inequality in \eqref{2.1} is proved. This completes the proof of Theorem \ref{th2.1}.
\end{proof}

\begin{corollary}Let $u: [a, b]\rightarrow \mathbb{R}$ be a positive function with $0\leq a < b$ and $u \in L_1 (a, b).$ If $u$ is a concave function on $[a, b],$ then the following inequalities for fractional integrals \eqref{I-1}
and \eqref{I-2} hold:
\begin{align*} u\left(\frac{a+b}{2}\right)\geq \frac{1-\alpha}{2\left(1-\exp\left(-\rho \right)\right)}\left[\mathcal{I}^\alpha_a u(b)+\mathcal{I}^\alpha_b u(a)\right]\geq \frac{u(a)+u(b)}{2}.
\end{align*}
\end{corollary}

\begin{remark}
For $\alpha \rightarrow 1,$ observe that
\begin{align*}\lim_{\alpha \rightarrow 1}\frac{1-\alpha}{2\left(1-\exp\left(-\rho \right)\right)}=\frac{1}{2(b-a)}.
\end{align*} Thus, Hermite-Hadamard inequality  \eqref{1.1} follows from Theorem \ref{th2.1} in the limit $\alpha \rightarrow 1.$
\end{remark}
\section{Hermite-Hadamard-Fej\'{e}r type inequality}\label{HHF}

\begin{theorem}\label{th3.1} Let $u:\,[a,b]\rightarrow \mathbb{R}$ be convex and integrable function with $a<b.$ If $w:\,[a,b]\rightarrow \mathbb{R}$ is nonnegative, integrable and symmetric with respect to $\frac{a+b}{2},$ that is,  $w(a+b-x)=w(x),$ then the following inequalities hold
\begin{multline}\label{3.1} u\left(\frac{a+b}{2}\right)\left[\mathcal{I}^{\alpha}_a w(b)+\mathcal{I}^{\alpha}_b w(a)\right] \leq \left[\mathcal{I}^{\alpha}_a \left(u w\right)(b)+\mathcal{I}^{\alpha}_b\left(u w\right)(a)\right]\\\leq \frac{u(a)+u(b)}{2}\left[\mathcal{I}^{\alpha}_a w(b)+\mathcal{I}^{\alpha}_b w(a)\right].
\end{multline}
\end{theorem}

\begin{proof}
Since $u$ is a convex function on $[a, b],$ we have  the inequality \eqref{2.3} for all $t\in [0; 1]$. Multiplying both sides of \eqref{2.3} by \begin{equation}\label{3.2}\exp\left(-\rho t \right)w\left((1-t)a+tb\right),\end{equation} and then integrating the resulting inequality with respect to $t$ over $[0, 1],$ we obtain
\begin{align*}2u\left(\frac{a+b}{2}\right)&\int\limits^1_0 \exp\left(-\rho t\right)w\left((1-t)a+tb\right) dt \\& \leq \int\limits^1_0 \exp\left(-\rho t\right) u\left(ta+(1-t)b\right) w\left((1-t)a+tb\right)dt\\& + \int\limits^1_0 \exp\left(-\rho t\right) u\left((1-t)a+tb\right) w\left((1-t)a+tb\right)dt\\& = \frac{1}{b-a}\int\limits^b_a \exp\left(-\frac{1-\alpha}{\alpha}(s-a)\right) u\left(a+b-s\right) w(s)ds\\& + \frac{1}{b-a}\int\limits^b_a \exp\left(-\frac{1-\alpha}{\alpha}(s-a)\right) u(s)w(s)ds\\&=\frac{1}{b-a}\int\limits^b_a \exp\left(-\frac{1-\alpha}{\alpha}(b-s)\right) u(s) w\left(a+b-s\right)ds \\&+ \frac{\alpha}{b-a}\mathcal{I}^\alpha_b \left[u(a)w(a)\right]=\frac{\alpha}{b-a}\left[\mathcal{I}^\alpha_a \left[u(a)w(a)\right]+\mathcal{I}^\alpha_b \left[u(a)w(a)\right]\right],
\end{align*}
that is,
\begin{multline*}2u\left(\frac{a+b}{2}\right)\int\limits^1_0 \exp\left(-\rho t\right)w\left((1-t)a+tb\right) dt \\ \leq \frac{\alpha}{b-a}\left[\mathcal{I}^\alpha_a \left[u(a)w(a)\right]+\mathcal{I}^\alpha_b \left[u(a)w(a)\right]\right].
\end{multline*}
Since $w$ is symmetric with respect to $\frac{a+b}{2},$ we have $$\mathcal{I}^\alpha_a w(b)=\mathcal{I}^\alpha_b w(a)=\frac{1}{2}\left[\mathcal{I}^\alpha_a w(b)+\mathcal{I}^\alpha_b w(a)\right].$$
Therefore, we have \begin{align*}u\left(\frac{a+b}{2}\right)\left[\mathcal{I}^\alpha_a w(b)+\mathcal{I}^\alpha_b w(a)\right] \leq\mathcal{I}^\alpha_a \left[w\left(b\right) u(b)\right]+\mathcal{I}^\alpha_b \left[w\left(a\right) u(a)\right].
\end{align*} This establishes the first inequality of Theorem \ref{th3.1}.

To prove the second inequality in \eqref{3.1}, we first notice that if $u$ is a convex
function, then, for all $t\in [0 1],$ it yields the inequality \eqref{2.4}. Multiplying both sides of \eqref{2.3} by \eqref{3.2} and integrating the resulting inequality with respect to $t$ over $[0, 1],$ we get \begin{align*}&\int\limits^1_0 \exp\left(-\rho t\right) u\left(ta+(1-t)b\right) w\left((1-t)a+tb\right)dt \\&+ \int\limits^1_0 \exp\left(-\rho t\right) u\left((1-t)a+tb\right)w\left((1-t)a+tb\right) dt \\& \leq \left[u(a)+u(b)\right]\int\limits^1_0 \exp\left(-\rho t\right)w\left((1-t)a+tb\right) dt.
\end{align*} In consequence, we obtain \begin{align*}\mathcal{I}^\alpha_a \left[w\left(b\right) u(b)\right]+\mathcal{I}^\alpha_b \left[w\left(a\right) u(a)\right] \leq \frac{u(a)+u(b)}{2} \left[\mathcal{I}^\alpha_a w(b)+\mathcal{I}^\alpha_b w(a)\right].\end{align*} Thus the proof of Theorem \ref{th3.1} is complete.
\end{proof}

\begin{corollary}Let $u:\,[a,b]\rightarrow \mathbb{R}$ be concave and integrable function with $a<b.$ If $w:\,[a,b]\rightarrow \mathbb{R}$ is nonnegative, integrable and symmetric with respect to $\frac{a+b}{2},$ that is,  $w(a+b-x)=w(x),$ then the following inequalities hold
\begin{align*}u\left(\frac{a+b}{2}\right)\left[\mathcal{I}^{\alpha}_a w(b)+\mathcal{I}^{\alpha}_b w(a)\right]& \geq \left[\mathcal{I}^{\alpha}_a \left(u w\right)(b)+\mathcal{I}^{\alpha}_b\left(u w\right)(a)\right]\\& \geq \frac{u(a)+u(b)}{2}\left[\mathcal{I}^{\alpha}_a w(b)+\mathcal{I}^{\alpha}_b w(a)\right].
\end{align*}
\end{corollary}

\begin{remark}
From Theorem \ref{th3.1} with $\alpha \to 1,$ we indeed have Hermite-Hadamard-Fej\'{e}r inequality  \eqref{1.2}.
\end{remark}

\section{Dragomir-Agarwal type inequality}\label{DA}

\begin{theorem}\label{th4.1}
Let $u:\, I\subseteq \mathbb{R}\rightarrow \mathbb{R}$ be a differentiable mapping on $I,a,b \in I.$ If $|u'|$ is convex on $[a,b],$ then the following inequality involving  fractional integrals \eqref{I-1}
and \eqref{I-2} holds: \begin{multline}\label{4.1}\left|\frac{u(a)+u(b)}{2}-\frac{1-\alpha}{2\left(1-\exp\left(-\rho\right)\right)} \left[\mathcal{I}^\alpha_a u(b)+\mathcal{I}^\alpha_b u(a)\right]\right|\\ \leq \frac{b-a}{2\rho}\tanh\left(\frac{\rho}{4}\right)\left(|u'(a)|+|u'(b)|\right).
\end{multline}
\end{theorem}
\begin{proof}
For $u'\in L_1(a,b),$ it is easy to find that
\begin{multline}\label{4.2}
\frac{u(a)+u(b)}{2}-\frac{1-\alpha}{2\left(1-\exp\left(-\rho\right)\right)}\left[\mathcal{I}^\alpha_b u(a)+\mathcal{I}^\alpha_a u(b)\right]\\ =\frac{b-a}{2\left(1-\exp\left(-\rho\right)\right)}\left\{\int\limits^1_0 \exp\left(-\rho t\right)u'\left(ta+(1-t)b\right)dt\right.\\ -\left.\int\limits^1_0 \exp\left(-\rho(1-t)\right)u'\left(ta+(1-t)b\right)dt\right\}.\end{multline}
Then, using \eqref{4.2} and the convexity of $|u'|,$ we obtain
\begin{align*}\left|\frac{u(a)+u(b)}{2}\right.& \left.-\frac{1-\alpha}{2\left(1-\exp\left(-\rho\right)\right)} \left[\mathcal{I}^\alpha_b u(a)+\mathcal{I}^\alpha_a u(b)\right]\right|\\& \leq \frac{b-a}{2}\int\limits^1_0 \frac{\left|\exp\left(-\rho t\right) -\exp\left(-\rho(1-t)\right)\right|}{1-\exp\left(-\rho\right)} \left|u'\left(ta+(1-t)b\right)\right|dt\\& \leq \frac{b-a}{2}\int\limits^1_0 \frac{\left|\exp\left(-\rho t\right) -\exp\left(-\rho(1-t)\right)\right|}{1-\exp\left(-\rho\right)} t\left|u'\left(a\right)\right|dt\\& +\frac{b-a}{2}\int\limits^1_0 \frac{\left|\exp\left(-\rho t\right) -\exp\left(-\rho(1-t)\right)\right|}{1-\exp\left(-\rho\right)} (1-t)\left|u'\left(b\right)\right|dt\\&= \frac{b-a}{2}\left|u'\left(a\right)\right|\int\limits^{\frac{1}{2}}_0 \frac{\exp\left(-\rho t\right) -\exp\left(-\rho(1-t)\right)}{1-\exp\left(-\rho\right)} t dt\\& + \frac{b-a}{2}\left|u'\left(a\right)\right|\int\limits^1_{\frac{1}{2}} \frac{\exp\left(-\rho(1-t)\right)-\exp\left(-\rho t\right)} {1-\exp\left(-\rho\right)} t dt\\& + \frac{b-a}{2}\left|u'\left(b\right)\right|\int\limits^{\frac{1}{2}}_0 \frac{\exp\left(-\rho t\right) -\exp\left(-\rho(1-t)\right)}{1-\exp\left(-\rho\right)} (1-t) dt\\&+ \frac{b-a}{2}\left|u'\left(b\right)\right|\int\limits^1_{\frac{1}{2}} \frac{\exp\left(-\rho (1-t)\right)- \exp\left(-\rho t\right)}{1-\exp\left(-\rho\right)} (1-t) dt\\& = \frac{b-a}{2\left(1-\exp\left(-\rho \right)\right)} \left[\left|u'\left(a\right)\right|\left(I_1+I_2\right) + \left|u'\left(b\right)\right|\left(I_3+I_4\right)\right].
\end{align*}
As a result, we get
\begin{multline}\label{4.3}\frac{u(a)+u(b)}{2}-\frac{1-\alpha}{2\left(1-\exp\left(-\rho\right)\right)}\left[\mathcal{I}^\alpha_b u(a)+\mathcal{I}^\alpha_a u(b)\right]\\ \leq \frac{b-a}{2\left(1-\exp\left(-\rho \right)\right)}\left[\left|u'\left(a\right)\right|\left(I_1+I_2\right) + \left|u'\left(b\right)\right|\left(I_3+I_4\right)\right],
\end{multline}
where
\begin{multline}\label{4.4}I_1=\int\limits^{\frac{1}{2}}_0 \left(\exp\left(-\rho t\right) -\exp\left(-\rho (1-t)\right)\right) t dt\\ =-\frac{\exp\left(-\frac{\rho}{2}\right)}{\rho}+ \frac{1}{\rho^2}\left(1 -\exp\left(-\rho\right)\right),
\end{multline}
\begin{multline}\label{4.5}I_2=\int\limits^1_{\frac{1}{2}} \left(\exp\left(-\rho(1-t)\right)-\exp\left(-\rho t\right)\right) t dt\\ =\frac{1}{\rho}\left(1 -\exp\left(-\frac{\rho}{2}\right)+\exp\left(-\rho\right)\right) -\frac{1}{\rho^2}\left(1 -\exp\left(-\rho \right)\right),
\end{multline}
\begin{multline}\label{4.6}I_3=\int\limits^{\frac{1}{2}}_0 \left(\exp\left(-\rho t\right) -\exp\left(-\rho (1-t)\right)\right) (1-t) dt\\ =-\frac{\exp\left(-\frac{\rho}{2}\right)} {\rho}+\frac{1}{\rho}\left(1+\exp\left(-\rho\right)\right) - \frac{1}{\rho^2}\left(1 -\exp\left(-\rho\right)\right)
\end{multline}
and
\begin{multline}\label{4.7}I_4=\int\limits^1_{\frac{1}{2}} \left(\exp\left(-\rho t\right)- \exp\left(-\rho (1-t)\right)\right) (1-t) dt\\ =-\frac{\exp\left(-\frac{\rho}{2}\right)}{\rho}+ \frac{1}{\rho^2}\left(1 -\exp\left(-\rho\right)\right).\end{multline}
Inserting the values of $I_i \, (i=1,2,3,4)$ given by  \eqref{4.4}-\eqref{4.7} in \eqref{4.3}, we obtain the inequality  \eqref{4.1}. This completes the proof.
\end{proof}

\begin{corollary}
Let $u:\, I\subseteq \mathbb{R}\rightarrow \mathbb{R}$ be a differentiable mapping on $I,a,b \in I.$ If $|u'|$ is concave on $[a,b],$ then the following inequality holds: \begin{multline*}\left|\frac{u(a)+u(b)}{2}-\frac{1-\alpha}{2\left(1-\exp\left(-\rho\right)\right)} \left[\mathcal{I}^\alpha_a u(b)+\mathcal{I}^\alpha_b u(a)\right]\right|\\ \geq \frac{b-a}{2\rho}\tanh\left(\frac{\rho}{4}\right)\left(|u'(a)|+|u'(b)|\right).
\end{multline*}
\end{corollary}

\begin{remark}
For $\alpha \rightarrow 1,$ we find that
\begin{align*}\lim_{\alpha \rightarrow 1}\frac{1-\alpha}{2\left(1-\exp\left(-\rho \right)\right)}=\frac{1}{2(b-a)},
\end{align*}
\begin{align*}\lim_{\alpha \rightarrow 1}\frac{b-a}{2\rho}\tanh\left(\frac{\rho}{4}\right)=\frac{b-a}{8}.
\end{align*} Thus we get Dragomir-Agarwal inequality  \eqref{1.3}  from Theorem \ref{th4.1} when  $\alpha \rightarrow 1.$
\end{remark}

\section{Pachpatte type inequalities}\label{P}
\begin{theorem}\label{th5.1}
Let $u$ and $w$ be real-valued, nonnegative and convex functions on $[a, b].$ Then the following inequalities involving  fractional integrals \eqref{I-1}
and \eqref{I-2} hold:
\begin{multline}\label{5.1}\frac{\alpha}{2(b-a)}\left[\mathcal{I}^\alpha_a\left(u(b)w(b)\right)+\mathcal{I}^\alpha_b\left(u(a)w(a)\right)\right]\\ \leq \left[u(a)w(a)+ u(b)w(b)\right]\frac{{\rho}^2-2 \rho+ 4- \left({\rho}^2+2 \rho+4\right)\exp\left(-\rho \right)}{2\rho^3}\\ +\left[u(a)w(b)+u(b)w(a)\right]\frac{\rho-2+ \exp\left(-\rho \right)\left(\rho+2\right)}{\rho^3},
\end{multline}
\begin{multline}\label{5.2}2u\left(\frac{a+b}{2}\right)w\left(\frac{a+b}{2}\right) \leq \frac{1-\alpha}{2\left(1-\exp\left(-\rho\right)\right)}\left[\mathcal{I}^{\alpha}_a u(b)w(b)+\mathcal{I}^{\alpha}_b u(a)w(a)\right]\\ + \left[u(a)w(a)+u(b)w(b)\right]\frac{\rho-2+ \exp\left(-\rho\right)\left(\rho +2\right)} {\rho^2\left(1-\exp\left(-\rho \right)\right)}\\ +\left[u(a)w(b)+u(b)w(a)\right]\frac{{\rho}^2-2 \rho+ 4- \left({\rho}^2+2 \rho+4\right)\exp\left(-\rho\right)} {2\rho^2\left(1-\exp\left(-\rho\right)\right)}.
\end{multline}
\end{theorem}

\begin{proof}
Since $u$ and $w$ are convex on $[a, b],$ then, for $\xi\in [0, 1],$ it follows from definition \ref{def1.1} that
\begin{align*}u\left(\xi a+(1-\xi)b\right)w\left(\xi a+(1-\xi)b\right)& \leq \xi^2 u(a)w(a)+ (1-\xi)^2u(b)w(b)\\&+\xi(1-\xi)\left[u(a)w(b)+u(b)w(a)\right]
\end{align*}
and
\begin{align*} u\left((1-\xi)a+\xi b\right)w\left((1-\xi)a+\xi b\right)& \leq (1-\xi)^2 u(a)w(a)+ \xi^2u(b)w(b)\\& +\xi(1-\xi)\left[u(a)w(b)+u(b)w(a)\right].
\end{align*}
Consequently, we have
\begin{equation}\label{5.3}\begin{aligned} u\left(\xi a+(1-\xi)b\right)w\left(\xi a+(1-\xi)b\right)& +u\left((1-\xi)a+tb\right)w\left((1-\xi)a+\xi b\right)\\& \leq (2\xi^2-2 \xi+1) \left[u(a)w(a)+ u(b)w(b)\right]\\& +2t(1-\xi)\left[u(a)w(b)+u(b)w(a)\right].\end{aligned}
\end{equation}
Multiplying both sides of inequality \eqref{5.3} by $\exp\left(-\rho \xi \right)$ and integrating the resulting inequality with respect to $\xi \in [0, 1],$ we obtain
\begin{multline*}\int\limits^1_0 \exp\left(-\rho \xi \right) u\left(\xi a+(1-\xi)b\right)w\left(\xi a+(1-\xi)b\right)d \xi \\ +\int\limits^1_0 \exp\left(-\rho \xi \right)u\left((1-\xi)a+\xi b\right)w\left((1-\xi)a+\xi b\right)d \xi \\ =\frac{\alpha}{b-a}\left[\mathcal{I}^\alpha_a\left(u(b)w(b)\right)+ \mathcal{I}^\alpha_b\left(u(a)w(a)\right)\right]\\ \leq \left[u(a)w(a)+ u(b)w(b)\right]\int\limits^1_0 \exp\left(-\rho \xi \right)(2\xi^2-2\xi+1)d \xi \\ +\left[u(a)w(b)+u(b)w(a)\right]\int\limits^1_0 \exp\left(-\rho \xi \right)2\xi(1-\xi)d \xi\\ =\left[u(a)w(a)+ u(b)w(b)\right]\frac{{\rho}^2-2 \rho+ 4- \left({\rho}^2+2 \rho+4\right)\exp\left(-\rho\right)}{\rho^3}\\ +2\left[u(a)w(b)+u(b)w(a)\right]\frac{\rho-2+ \exp\left(-\rho\right)\left(\rho+2\right)}{\rho^3}.
\end{multline*}

So
\begin{multline*}\frac{\alpha}{2(b-a)}\left[\mathcal{I}^\alpha_a\left(u(b)w(b)\right)+\mathcal{I}^\alpha_b\left(u(a)w(a)\right)\right]\\ \leq \left[u(a)w(a)+ u(b)w(b)\right]\frac{{\rho}^2-2 \rho+ 4- \left({\rho}^2+2 \rho+4\right)\exp\left(-\rho \right)}{2\rho^3}\\ +\left[u(a)w(b)+u(b)w(a)\right] \frac{\rho-2+\exp\left(-\rho \right)\left(\rho+2\right)}{\rho^3},
\end{multline*} which completes the proof of \eqref{5.1}.

Next we establish the inequality \eqref{5.2}. Again using convexity of the functions $u$ and $v$ on $[a,b],$ we have
\begin{align*}& u\left(\frac{a+b}{2}\right)w\left(\frac{a+b}{2}\right)\\& = u\left(\frac{\xi a+(1-\xi)b}{2}+\frac{(1-\xi)a+\xi b}{2}\right) w\left(\frac{\xi a+(1-\xi)b}{2}+\frac{(1-\xi)a+\xi b}{2}\right)\\& \leq \left(\frac{u\left(\xi a+(1-\xi)b\right)+u\left((1-\xi)a+\xi b\right)}{2}\right) \left(\frac{w\left(ta+(1-\xi)b\right)+w\left((1-\xi)a+\xi b\right)}{2}\right)\\& \leq\frac{u\left(\xi a+(1-\xi)b\right)w\left(\xi a+(1-\xi)b\right)}{4} +\frac{u\left((1-\xi)a+\xi b\right)w\left((1-\xi)a+\xi b\right)}{4}\\& +\frac{\xi(1-\xi)}{2}\left[u(a)w(a)+u(b)w(b)\right] +\frac{(2\xi^2-2\xi+1)}{4}\left[u(a)w(b)+u(b)w(a)\right].
\end{align*}
Thus
\begin{equation}\label{5.4} \begin{aligned} &u\left(\frac{a+b}{2}\right)w\left(\frac{a+b}{2}\right)\\ & \leq\frac{u\left(\xi a+(1-\xi)b\right)w\left(\xi a+(1-\xi)b\right)}{4} +\frac{u\left((1-\xi)a+\xi b\right)w\left((1-\xi)a+\xi b\right)}{4}\\& +\frac{t(1-\xi)}{2}\left[u(a)w(a)+u(b)w(b)\right] +\frac{(2\xi^2-2\xi+1)}{4}\left[u(a)w(b)+u(b)w(a)\right].
\end{aligned}\end{equation} Multiplying both sides of \eqref{5.4} by $\exp\left(-\rho \xi \right)$ and then integrating the resulting inequality with respect to $t \in [0, 1],$ we have
\begin{align*}&\frac{1-\exp\left(-\rho \right)} {\rho}u\left(\frac{a+b}{2}\right)w\left(\frac{a+b}{2}\right)\\& \leq \int\limits^1_0 \exp\left(-\rho \xi \right)\frac{u\left(\xi a+(1-\xi)b\right)w\left(\xi a+(1-\xi)b\right)}{4}d \xi\\& +\int\limits^1_0 \exp\left(-\rho \xi \right)\frac{u\left((1-\xi)a+\xi b\right)w\left((1-\xi)a+\xi b\right)}{4}d \xi\\& +\int\limits^1_0 \exp\left(-\rho \xi\right) \frac{\xi(1-\xi)}{2}\left[u(a)w(a)+u(b)w(b)\right]d \xi\\& +\int\limits^1_0 \exp\left(-\rho \xi \right)\frac{2\xi^2-2\xi+1}{4}\left[u(a)w(b)+u(b)w(a)\right]d \xi\\& =\frac{\alpha}{4(b-a)}\left[\mathcal{I}^{\alpha}_a u(b)w(b)+\mathcal{I}^{\alpha}_b u(a)w(a)\right]\\& + \left[u(a)w(a)+u(b)w(b)\right]\frac{\rho-2+ \exp\left(-\rho\right)\left(\rho+2\right)}{2\rho^3}\\ &+\left[u(a)w(b)+u(b)w(a)\right]\frac{{\rho}^2-2 \rho+ 4- \left({\rho}^2+2 \rho+4\right)\exp\left(-\rho\right)}{4\rho^3},
\end{align*}
which can alternatively be written as
\begin{align*}& u\left(\frac{a+b}{2}\right)w\left(\frac{a+b}{2}\right)\\& \leq\frac{1-\alpha}{4\left(1-\exp\left(-\rho\right)\right)}\left[\mathcal{I}^{\alpha}_a u(b)w(b)+\mathcal{I}^{\alpha}_b u(a)w(a)\right]\\& + \left[u(a)w(a)+u(b)w(b)\right]\frac{\rho-2+ \exp\left(-\rho\right)\left(\rho+2\right)} {2\rho^2\left(1-\exp\left(-\rho\right)\right)}\\ &+\left[u(a)w(b)+u(b)w(a)\right]\frac{{\rho}^2-2 \rho+ 4- \left({\rho}^2+2 \rho+4\right) \exp\left(-\rho\right)}{4\rho^2\left(1-\exp\left(-\rho\right)\right)}.
\end{align*}
This completes the proof.
\end{proof}

\begin{corollary}
Suppose that  $u$ and $w$ are real-valued, nonnegative and concave functions on $[a, b].$ Then the following inequalities hold:
\begin{multline*}\frac{\alpha}{2(b-a)}\left[\mathcal{I}^\alpha_a\left(u(b)w(b)\right)+ \mathcal{I}^\alpha_b\left(u(a)w(a)\right)\right]\\ \geq \left[u(a)w(a)+ u(b)w(b)\right]\frac{{\rho}^2-2 \rho+ 4- \left({\rho}^2+2 \rho+4\right)\exp\left(-\rho\right)}{2\rho^3}\\ +\left[u(a)w(b)+u(b)w(a)\right]\frac{\rho-2+ \exp\left(-\rho \right)\left(\rho+2\right)}{\rho^3},
\end{multline*}
\begin{multline*} 2 u\left(\frac{a+b}{2}\right)w\left(\frac{a+b}{2}\right) \geq \frac{1-\alpha}{2\left(1-\exp\left(-\rho\right)\right)}\left[\mathcal{I}^{\alpha}_a u(b)w(b)+\mathcal{I}^{\alpha}_b u(a)w(a)\right]\\ + \left[u(a)w(a)+u(b)w(b)\right]\frac{\rho-2+ \exp\left(-\rho\right)\left(\rho+2\right)} {\rho^2\left(1-\exp\left(-\rho\right)\right)}\\ +\left[u(a)w(b)+u(b)w(a)\right]\frac{{\rho}^2-2 \rho+ 4- \left({\rho}^2+2 \rho+4\right)\exp\left(-\rho\right)} {2\rho^2\left(1-\exp\left(-\rho\right)\right)}.
\end{multline*}
\end{corollary}

\begin{remark} Using the limiting values
\begin{align*}\lim_{\alpha \rightarrow 1}\frac{1-\alpha}{2\left(1-\exp\left(-\rho\right)\right)}=\frac{1}{2(b-a)}, \, \,
\lim_{\alpha \rightarrow 1}\frac{\rho-2+\exp\left(-\rho\right)\left(\rho+2\right)}{\rho^3}=\frac{1}{6},
\end{align*}
\begin{align*}
\lim_{\alpha \rightarrow 1}\frac{{\rho}^2-2 \rho+ 4- \left({\rho}^2+2 \rho+4\right)\exp\left(-\rho\right)} {2\rho^2\left(1-\exp\left(-\rho\right)\right)}=\frac{1}{3},
\end{align*} we obtain Pachpatte inequalities \eqref{1.4} and \eqref{1.5} from Theorem \ref{th5.1} when $\alpha \to 1.$
\end{remark}

\section*{Discussions and Conclusions} We obtained the generalization of the Hermite-Hadamard, HermiteHadamard-Fej\'{e}r, Dragomir-Agarwal and Pachpatte type inequalities for new fractional integral operators with exponential kernel. As an immediate consequence of the results derived in this paper, one can obtain similar inequalities for the following fractional integrals with Mittag-Leffler nonsingular kernel:\\
\begin{equation*}
\mathfrak{I}^\alpha_a u(x)=\frac{1}{\alpha}\int\limits^x_a E_{\alpha,1}\left(-\frac{1-\alpha}{\alpha}(x-s)^\alpha\right) u(s)ds,\, x>a
\end{equation*}
and
\begin{equation*}
\mathfrak{I}^\alpha_b u(x)=\frac{1}{\alpha}\int\limits^b_x E_{\alpha,1}\left(-\frac{1-\alpha}{\alpha}(s-x)^\alpha\right) u(s)ds,\, x<b
\end{equation*}
for $f\in L_1(a,b)$ and $\alpha\in (0,1).$ Here $E_{\alpha,\mu}\left(z\right)$ is the
Mittag-Leffler type function: $$E_{\alpha,\mu}\left(z\right)=\sum\limits_{k=0}^{\infty}
\frac{z^k}{\Gamma\left(\alpha k+\mu\right)}.$$
Moreover, we believe that the present work would serve as a strong motivation for the fellow researchers to enhance/enrich similar known literature on the related topics.

\section*{Acknowledgements}
The research of Torebek is financially supported by a grant from the Ministry of Science and Education of the Republic of Kazakhstan (Grants No.AP05131756). The authors gratefully acknowledge the referee for his/her useful comments that led to the improvement of the original manuscript.


\begin{thebibliography}{99}

\bibitem{BPP16} D. Baleanu, S.D. Purohit, J.C. Prajapati, Integral inequalities involving generalized Erdelyi-Kober fractional integral operators,  Open Mathematics {\bf 14}, 1(2016), 89-99.

\bibitem{C16} F. Chen,  Extensions of the Hermite-Hadamard inequality for convex functions via fractional integrals, J. Math. Inequal. {\bf 10}, 1(2016), 75-81.

\bibitem{CK17} H. Chen, U.N. Katugampola, Hermite-Hadamard and Hermite-Hadamard-Fejer type inequalities for generalized fractional integrals,  J. Math. Anal. Appl. {\bf 446}, 2(2017), 1274-1291.

\bibitem{DA98} S.S. Dragomir, R.P. Agarwal, Two inequalities for differentiable mappings and applications to special means of real numbers and to trapezoidal formula, Appl. Math. lett. {\bf 11}, 5 (1998), 91-95.

\bibitem{DP00} S.S. Dragomir, C.E.M. Pearce, Selected topics on Hermite-Hadamard inequalities and applications, RGMIA Monographs, Victoria University, 2000.

\bibitem{F06} L. Fej\'{e}r,  Uberdie Fourierreihen, II, Math., Naturwise. Anz Ungar. Akad.Wiss, {\bf 24}, (1906), 369-390 (in Hungarian).

\bibitem{H1893} J. Hadamard, Etude sur les proprietes des fonctions entieres et en particulier d'une fonction considree par Riemann, J. Math. Pures et Appl. {\bf 58},  (1893), 171-215.

\bibitem{H1883} Ch. Hermite, Sur deux limites d'une integrale definie, Mathesis {\bf 3}, (1883), 82.

\bibitem{HYT14} S.R. Hwang, S.Y. Yeh, K.L. Tseng, Refinements and similar extensions of Hermite-Hadamard inequality for fractional integrals and their applications,  Appl. Math. Computat. {\bf 249}, (2014), 103-113.

\bibitem{ITM16} I. I\c{s}can, S. Turhan, S. Maden, Some Hermite-Hadamard-Fejer type inequalities for harmonically convex functions via fractional integral, New Trends Math. Sci. {\bf 4}, 2(2016), 1-10.

\bibitem{I16} I. Iscan, On generalization of different type inequalities for harmonically quasi-convex functions via fractional integrals, Appl. Math. Computat. {\bf 275}, (2016),  287-298.

\bibitem{JS16} M. Jleli, B. Samet, On Hermite-Hadamard type inequalities via fractional integrals of a function with respect to another function,   J. Nonlinear Sci. Appl. {\bf 9}, 3(2016),  1252-1260.

\bibitem{P03} B.G. Pachpatte, On some inequalities for convex functions, RGMIA Res. Rep. Coll. {\bf 6} (E), 2003.

\bibitem{PPT92} J.E. Pe\v{c}ari\'{c}, F. Proschan, Y.L. Tong, Convex Functions, Partial Orderings and Statistical Applications, Academic Press, Boston, 1992.

\bibitem{SSYB13} M.Z. Sarikaya, E. Set,  H. Yaldiz,  N. Ba\c{s}ak,  Hermite-Hadamard's inequalities for fractional integrals and related fractional inequalities, Math. Comput. Modelling {\bf 57}, 9-10 (2013), 2403-2407.

\bibitem{WLFZ12} J. Wang, X. Li, M. Fe\v{c}kan, Y. Zhou, Hermite-Hadamard-type inequalities for Riemann-Liouville fractional integrals via two kinds of convexity, Appl. Anal. {\bf 92}, 11(2012), 2241-2253.

\bibitem{ZW13} Y. Zhang,  J. Wang,  On some new Hermite-Hadamard inequalities involving Riemann-Liouville fractional integrals, J. Inequal. Appl. {\bf 2013}, 220(2013), 27 pp.

\end{thebibliography}
\end{document}